\documentclass[a4paper,11pt]{article}
\usepackage[T1]{fontenc}
\usepackage{lmodern}
\usepackage{amsthm,amsmath,amsfonts,amssymb,bbm,mathrsfs,mathtools}
\usepackage{enumitem}
\usepackage{url}
\usepackage{dsfont}
\usepackage{appendix}
\usepackage{color}
\usepackage{graphicx}
\usepackage[dvipsnames]{xcolor}
\usepackage{float}
\usepackage{braket}
\usepackage{tikz}
\usetikzlibrary{decorations}
\usetikzlibrary{decorations.pathmorphing}
\usetikzlibrary{decorations.pathreplacing}
\usetikzlibrary{decorations.shapes}
\usetikzlibrary{decorations.text}
\usetikzlibrary{decorations.markings}
\usetikzlibrary{decorations.fractals}
\usetikzlibrary{decorations.footprints}
\usepackage{authblk}
\usepackage[colorlinks=true, linkcolor=blue, urlcolor=black, citecolor=blue, pdfstartview=FitH]{hyperref}
\usepackage[english]{babel}
\usepackage{caption,subfigure}
\usetikzlibrary{shapes}
\usetikzlibrary{patterns}
\usepackage[top=3cm, bottom=3cm, left=2.3cm, right=2.3cm]{geometry}

\makeatletter

\@addtoreset{equation}{section}
\makeatother

\setlist[enumerate,1]{label=(\roman*),font=\normalfont}

\newcommand{\ap}{\alpha_{\text{\fontsize{4}{5}\selectfont +}}^\varepsilon}
\newcommand{\am}{\alpha_{\text{\fontsize{6}{7}\selectfont -}}^\varepsilon}
\newcommand{\be}{\beta}
\newcommand{\bh}{\frac{\beta^2}{2}}
\newcommand{\cb}{\kappa_\beta}

\newcommand{\e}{\mathrm{e}}

\newcommand{\eb}{\e^{\beta X_u-n \cb}}
\newcommand{\ebt}{\e^{\beta X_t(u)-\frac{\beta^2}{2}t}}
\newcommand{\ebsa}{\e^{\beta (a+X_s(u))-\frac{\beta^2}{2}s}}
\newcommand{\ebta}{\e^{\beta (a+X_t(u))-\frac{\beta^2}{2}t}}
\newcommand{\ebsu}{\e^{\beta X_s(u)-\frac{\beta^2}{2}s}}
\newcommand{\ebsv}{\e^{\beta X_s(v)-\frac{\beta^2}{2}s}}
\newcommand{\ebuv}{\e^{\beta (X_s(u)+X_s(v))-\beta^2 s}}

\newcommand{\ep}{\varepsilon}

\newcommand{\es}{\mathbb{E}}
\newcommand{\de}{\coloneqq}
\newcommand{\ds}{\mathrm{d}s}
\newcommand{\dt}{\mathrm{d}t}
\newcommand{\dx}{\mathrm{d}x}
\newcommand{\fn}{\mathcal F_n}
\newcommand{\ind}{\mathds{1}}

\newcommand{\nt}{\mathcal N_t}
\newcommand{\p}{\mathbb{P}}
\newcommand{\pne}{\mathsf \Pi^\varepsilon_n}

\newcommand{\re}{\mathbb{R}}

\newcommand{\summ}{\sum_{u\in\nt}}

\newcommand{\sun}{\sum_{|u|=n}}
\newcommand{\ft}{\partial \mathbb T}

\newcommand{\vne}{\mathsf V_n^\varepsilon}
\newcommand{\wb}{W_\beta}
\newcommand{\zb}{Z_\beta}

\newcommand{\znb}{Z_n(\beta)}

\theoremstyle{plain}
\newtheorem{thm}{Theorem}[section]
\newtheorem{prop}[thm]{Proposition}
\newtheorem{lem}[thm]{Lemma}

\theoremstyle{definition}

\theoremstyle{remark}
\newtheorem{rem}[thm]{Remark}

\usepackage{anyfontsize}

\title{The left tail of the subcritical derivative martingale in a branching random walk}
\author{
	 Benjamin \textsc{Bonnefont}\thanks{Sorbonne Universit\'e, LPSM, \texttt{benjamin.bonnefont@sorbonne-universite.fr}}

	Vincent \textsc{Vargas}\thanks{Universit\'e de Genève, Section de Mathématiques, \texttt{vincent.vargas@unige.ch}} 
}

\begin{document}
	
\maketitle

\begin{abstract}
Motivated by the study of the quantum Mabuchi theory \cite{lacoinrhodesvargas22}, we obtain in this work a sharp estimate on the left tail of the distribution of the so-called derivative martingale in the $L^4$ phase.
\end{abstract}




\section{Introduction}


Branching random walks (BRW) have a long history in probability theory and its applications. Given a real parameter $\beta$ (the inverse temperature in the language of statistical physics), one can associate to a BRW a natural random measure called a multiplicative cascade. Multiplicative cascades were introduced by Mandelbrot in \cite{mandelbrot74} as a toy model for 
Gaussian multiplicative chaos (GMC), a random measure which models energy dissipation in a turbulent flow. The precise mathematical construction of GMC was established by Kahane in the landmark paper \cite{kahane85} following the works on turbulence by Kolmogorov-Obukhov \cite{kolmogorov62,obukhov62} and  Mandelbrot \cite{mandelbrot72}. In both aforementionned models, the main object of interest (a random measure) is obtained as the limit of a positive martingale and it is also very natural to inquire on the existence and the properties of the derivatives of these martingales which form yet another sequence of martingales. In the case of GMC theory, these derivatives appear as a crucial ingredient in the construction of the quantum Mabuchi theory \cite{lacoinrhodesvargas22} or as so-called logarithmic fields in conformal field theory \cite{zamolodchikov2004higher}. In this paper, we will work within the simplified framework of the Gaussian BRW as motivated by \cite{lacoinrhodesvargas22}, we will give the first sharp estimates on the left tail of the derivative.\\

We consider the case of a BRW with binary splitting and independent standard Gaussian increments. The process is indexed by the binary tree $\mathbb T=\cup_{n\geq 0}\, \{0,1\}^n$ where by convention $\{0,1\}^0=\{\varnothing\}$. For $u\in\mathbb T$, let us denote $|u|$ the length of $u$, $v\leq u $ if $v$ is an ancestor of $u$ and $u\wedge v$ the \textit{length} of the last common ancestor of $u$ and $v$. The value of the BRW at $u$ is given by
\[
X_u \de \sum_{\varnothing <v\leq u} G_v,
\]
where $(G_v)_{v\in\mathbb T}$ are i.i.d. standard Gaussian random variables\footnote{Here and hereafter, we adopt the classical conventions $\sum_\emptyset = 0$ and $\prod_\emptyset = 1$.}. The natural filtration of the process is given by $\mathcal F_n = \sigma(X_u,|u|\leq n)$. We can then define the additive martingale with real parameter $\beta$
\[
W_n(\beta)\coloneqq    \frac{1}{2^n}\sun \e^{\beta X_u- \frac{\beta^2}{2}n},
\]
which is a positive martingale with respect to $(\fn)$ and thus converges almost surely to a limit $\wb$. It is well known that this limit is non trivial if and only if $\beta\in(-\beta_c,\beta_c)$ where $\beta_c= \sqrt{2 \log 2}$. {\bf In the sequel, we will only consider the case $\beta\in[0,\beta_c)$}. The (almost surely positive) random variable $\wb$ is very well understood. For instance the right tail is known with high precision and follows a power law decay, see for instance \cite{guivarc1990extension}. In particular, $\wb$ is in $L^p$ if and only if $p < \left ( \frac{\beta_c}{\beta} \right )^2$.   \\

Since for all $\beta$, $W_n(\beta)$  is a martingale, all higher order derivatives with respect to $\beta$ are also martingales. In this paper, we will focus on the first derivative and hence study the following martingale 
\[
\znb\coloneqq  \dfrac{1}{2^n} \sun e^{\beta X_u-\frac{\beta^2}{2}n}\big(X_u-\beta n\big).
\]
\noindent
Thanks to the almost sure uniform convergence on (complex) compact sets included in the domain $\Lambda$ of figure \ref{domain} obtained in \cite{biggins92}, the function $\beta \mapsto \wb$ is almost surely analytical on $\Lambda$,
\begin{figure}[H]
	\centering
		
		\begin{tikzpicture}[scale=1.7]
			\draw [help lines] (-3,0) -- (3,0);
			\draw [help lines] (0,-2) -- (0,2);
			\draw [help lines] (2.35482/2,-0.05) -- (2.35482/2,0.05);
			\draw [help lines] (2.35482,0.05) -- (2.35482,0.05);
			
			\draw [dashed] (1.17741,1.17741) -- (1.17741,0);
			
			\draw [very thick, color = NavyBlue] (1.17741,1.17741) arc (45:135:1.6651);
			\draw [very thick, color = NavyBlue] (-1.17741,-1.17741) arc (-135:-45:1.6651);
			\draw [very thick, color = NavyBlue] (-1.17741,1.17741) -- (-2.35482,0);
			\draw [very thick, color = NavyBlue] (-1.17741,-1.17741) -- (-2.35482,0);
			\draw [very thick, color = NavyBlue] (1.17741,1.17741) -- (2.35482,0);
			\draw [very thick, color = NavyBlue] (1.17741,-1.17741) -- (2.35482,0);

			\filldraw [color = black] (2.35482,-0.0) circle (1pt)	node [below]{$\,\beta_c$};	
			\filldraw [color = black] (1.3,0.05) circle (0pt)	node [below]{ $\frac{\beta_c}{2}$};
			\draw (1.8,0.05) node [below]{ $\frac{\beta_c}{\sqrt2}$};	
			
			\filldraw [thick, color = black] (-2.35482,0) circle (1pt);
			\filldraw [thick, color = black] (1.17741,1.17741) circle (1pt);
			\filldraw [thick, color = black] (1.17741,-1.17741) circle (1pt);
			\filldraw [thick, color = black] (-1.17741,1.17741) circle (1pt);
			\filldraw [thick, color = black] (-1.17741,-1.17741) circle (1pt);
			\draw [thick, color = black] (0,0) circle (1.6651);
			\draw [thick, color = Emerald] (0,0) circle (1.17741);

		\end{tikzpicture}
		
		\caption{\small The domain of convergence of the additive martingale $\Lambda$ is delimited by the blue curves. The black circle delimits the $L^2$ phase and the inner circle the $L^4$ phase.}
		
		\label{domain}
\end{figure}
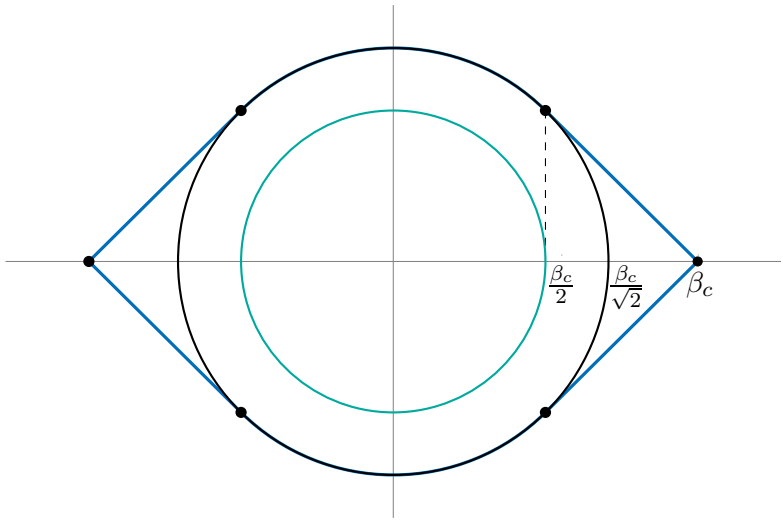
\noindent
and $\znb$ converges also almost surely to a non trivial limit $\zb=\partial_\be \wb$ when $\beta \in (0,\beta_c)$\footnote{In fact, the convergence holds in the whole domain of convergence.}. Contrary to the critical case, the limit $Z_\be$ is signed and has a highly non symmetrical distribution. Indeed, the right tail $Z_\be$ is expected to behave similarly to the right tail of $\wb$ up to logarithmic corrections whereas the left tail is expected to be very thin. The behavior of the left tail was conjectured for general  GMC measures in \cite[Conjecture 1]{lacoinrhodesvargas22}. The purpose of this paper is precisely to prove this conjecture within the framework of the Gaussian BRW. In this context, the conjecture can be rephrased as follows: for $\beta\in(0,\beta_c)$, there exist constants $c,c',C,C'>0$ such that for all $x\geq0$  
\[
Ce^{-cx^\gamma } \leq \p(\zb<-x) \leq C'e^{-c'x^\gamma },
\]
where 
\[
\gamma\de   \left(\frac{\beta_c}{\beta}\right)^2.
\]
We give a proof for the lower bound in the whole phase $\beta\in(0,\beta_c)$ and one for the upper bound in the so-called $L^4$ phase where $\beta\in(0,\beta_c/2)$.\\

\subsection{Results}
We will divide the presentation of our results in two parts. First we give a lower bound on the left tail in the full regime $\beta \in (0, \beta_c)$.
\begin{thm}\label{lowerTh}
	For $\beta\in\left(0,\beta_c\right)$, there exist $c$ and $C>0$ such that for $x\geq0$
	\[
	\p(\zb<-x) \geq C\e^{-cx^\gamma }.
	\]
\end{thm}
\noindent
Second we give an upper bound on the left tail in the regime $\beta\in(0,\beta_c/2)$.

\begin{thm}\label{lowerThbis}
	For $\beta\in\left(0,\frac{\beta_c}{2}\right)$, there exist $c'$ and $C'>0$ such that for $x\geq0$
	\[
	\p\left(\zb<-x\right) \leq C'\e^{-c'x^\gamma }.
	\]
\end{thm}
\noindent
We will make two remarks on the limiting cases $\beta=0$ and $\beta= \beta_c$ where much more is known and the behaviour of the left tail is quite different.
\begin{rem}
In the case $\beta = 0$, one can check that $Z_0=\lim\frac{1}{2^n}\sun X_u$ is a standard Gaussian random variable. 
\end{rem}

\begin{rem}
The case $\beta =  \beta_c$ has been thoroughly studied in the literature. In this case, the variable $W_{\beta_c}$ is trivial; however $Z_{\beta_c}$ is non trivial, negative almost surely and the left tail follows the power law
\[
	 \quad \p(Z_{\beta_c}<-x) \sim \frac{C}{x},\quad\text{ as } x\rightarrow+\infty,
	\]
for some $C>0$, see \cite{buraczewski2009}.
\end{rem}


\section{Lower bound for the left tail: proof of Theorem \ref{lowerTh}}


This section is devoted to proving Theorem \ref{lowerTh}. 
\subsection{Heuristics}

A quick inspection of $Z_n(\beta)$ as a function of $(X_u)_{|u|=n}$, namely
\[
Z_n(\beta)=\frac{1}{2^n} \sun e^{\beta X_u-\frac{\beta^2}{2} n }\big(X_u-\beta n\big),
\]
reveals that
\[
\text{ess}\inf Z_n(\beta)=-\frac{1}{\be \e} \e^{\frac{\beta^2}{2}n}=:-m_n.
\]
This value is achieved when all the particles are located at $\beta n-\frac{1}{\beta}$. This suggests the following scenario to get a lower bound. Put all the particles around $\beta n-\frac{1}{\beta}$ to minimize $Z_n(\beta)$ and then show that, on this event, the rest $Z_\beta-Z_n(\beta)$ is negligible using the decomposition of the derivative martingale
\[
\zb = \znb+R_n(\beta),
\]
with
\begin{equation}\label{decomp}
R_n(\beta)=\frac{1}{2^n} \sun e^{\beta X_u-\frac{\beta^2}{2}n} \zb^u+\frac{1}{2^n} \sun e^{\beta X_u-\frac{\beta^2}{2}n}  \big(X_u-\beta n\big) \big(\wb^u-1\big)
\end{equation}
and $(\zb^u,\wb^u)$ are independent copies of $(Z_\be,W_\be)$ independent of $\mathcal F_n$.\\

\noindent
To see why this works, suppose for instance that we are in the $L^2$ phase ($\beta<\beta_c/\sqrt2$). In this case, $Z^u$ and $W^u-1$ are centered with  finite variance, thus the central limit theorem ensures that $\sun Z_\beta^u$ fluctuates as $2^{n/2}$. If the $X_u$ are all located near $\beta n-\frac{1}{\beta}$, we have
\[
\frac{1}{2^n}\sun e^{\beta X_u-\frac{\beta^2}{2}n}  \zb^u ~\simeq~ \left( \frac{\e^{\bh }}{\sqrt 2}\right) ^n\frac{1}{2^{\frac{n}{2}}}\sun \zb^u
\]
which is exponentially small thanks to the condition $\beta< \beta_c$.\\
The cost for placing all the particles around $\beta n-\frac{1}{\beta}$ is of order $\exp(-2^n)$ as we will see and noting that $2^n = \left(\e^{\frac{\beta^2}{2}n}\right)^\gamma$ shows that we obtain the expected order. 
\begin{rem}
The probability that all the particles are near $\be n-\frac{1}{\be}$ at time $n$ in the i.i.d. case is
\begin{align*}
\p\Big(\mathcal N(0,n) = \be n-\frac{1}{\be}+O(1) \Big)^{2^n}
& = \e^{-\bh n2^n+ o(n2^n)}.
\end{align*}

We will see in the proof that the extra $n$ term in the exponential comes from the i.i.d. assumption and it will disappear for the correlated process $(X_u, |u|=n)$.
\end{rem}
\medskip
\noindent
Following the above heuristic, we start by giving a lower bound for the left tail of $\znb$ and then a lower bound for the left tail of $\zb$ by showing that  $\znb$ and  $\zb$ are close enough when the aforementionned scenario occurs.

\subsection{Lower bound for the left tail of $\znb$}\label{boxsce}

The process $(X_u, |u|=n)$ is a Gaussian process with covariance given by
$\text{Cov}(X_u,X_v)=u\wedge v$, where $u\wedge v$ is the length of the last common ancestor of $u$ and $v$. By ordering the $n$-th generation with the lexicographical order, this covariance gives a covariance matrix $\Sigma_n\in\mathcal M_{2^n}(\mathbb R)$. The first ones are
\[
\Sigma_1 = I_2,~~
\Sigma_2 = \begin{pmatrix} 2 & 1 & 0 & 0 \\ 1 & 2 & 0 & 0 \\ 0 & 0 & 2 & 1 \\ 0 & 0 & 1 & 2 \\ \end{pmatrix},~~
\Sigma_{n+1} = \left (\begin{array}{c|c}
J_n+\Sigma_n & 0 \\
\hline
0 & J_n+\Sigma_n
\end{array}\right),
\]
where $J_n\in\mathcal M_{2^n}(\mathbb R)$ have all entries equal to $1$. The recurrence relation above enables to compute the determinant of $\Sigma_n$ which is $(2^n-1) \prod_{k=0}^{n-1}(2^{n-k}-1)^{2^k}$. And the terms in the product are exactly the eigenvalues of $\Sigma_n$ as can be seen by using the same recursion for the characteristic polynomials.\\

\noindent
Now, recall that $-m_n=-\frac{1}{\e \beta}\exp\left(\bh n\right)$ is the essential infimum of $\znb$ and for $\ep\in(0,1)$, let
\[
\vne \coloneqq \left\{x=(x_u)_{|u|=n}\in\re^{2^n} : \frac{1}{2^n}\sun \e^{\beta x_u-n \frac{\beta^2}{2}}(x_u-\beta n)<-(1-\varepsilon)m_n \right\},
\]
so that the following equality holds:
\begin{align*}
\p\left(\znb<-(1-\varepsilon)m_n\right)
&=\p\big((X_u)_{|u|=n}\in\vne\big)\\
&=\frac{1}{\sqrt{2\pi}^{\,2^n}\sqrt{\det \Sigma_n}} \int_{\vne} \e^{-\frac{1}{2}x^t\Sigma_n^{-1}x} \dx.
\end{align*}
If $0<\am<1<\ap$ are the two solutions of the equation $\alpha\e^{-\alpha} = \frac{1}{\e}(1-\varepsilon)$, it is not hard to see that
$\pne\coloneqq\left[\beta n-\frac{\ap}{\be}, \beta n-\frac{\am}{\be}\right]^{2^n}\subset \vne$. Then
\begin{align*}
\p\big(\znb<-(1-\varepsilon)m_n\big) &\geq\frac{1}{\sqrt{2\pi}^{\,2^n}\sqrt{\det \Sigma_n}} \int_{\pne} \e^{-\frac{1}{2}x^t\Sigma_n^{-1}x} \dx\\
&\geq\frac{1}{\sqrt{2\pi}^{\,2^n}\sqrt{\det \Sigma_n}} \,\lambda\big(\pne\big)\, \exp\left(-\frac{1}{2}\,\sup_{x \in \pne} x^t\Sigma_n^{-1}x\right).
\end{align*}
It is an exercise to show that $\log \det \Sigma_n = \theta 2^n-2\log 2+o(1)$ for some positive $\theta\simeq 0.9458$, the computations are done in Lemma \ref{det}.
Notice also that the biggest eigenvalue of $\Sigma_n^{-1}$ is $1$. Therefore if $x\in\pne$, by writing $x=\left(\beta n-\frac{1}{\beta}\right)\ind+h$ where $\ind$ is the vector with all coordinates equal to one, one has
\begin{align*}
x^t\Sigma_n^{-1}x&=\left(\beta n-\frac{1}{\beta}\right)^2 \ind^t\Sigma_n^{-1}\ind+2\left(\beta n-\frac{1}{\beta}\right)\ind^t\Sigma_n^{-1}h+h^t\Sigma_n^{-1}h\\
&=\left(\beta n-\frac{1}{\beta}\right)^2\frac{2^n}{2^n-1}+2\left(\beta n-\frac{1}{\beta}\right)\frac{1}{2^n-1}\sum _i h_i+h^t\Sigma_n^{-1}h\\
&\leq\left(\frac{\ap-1}{\beta}\right)^2 2^n+O(n^2),
\end{align*}
where we used the fact that $\ap-1>1-\am$.
And one obtains
\begin{align*}
\p\big(\znb<-(1-\varepsilon)m_n\big) &\geq\e^{-\kappa_\varepsilon 2^n +O(n^2)},
\end{align*}
where $\kappa_\varepsilon\coloneqq \frac{1}{2}\left(\log 2\pi+\theta\right)-\log\frac{\ap-\am}{\beta}+\frac{1}{2}\left(\frac{\ap-1}{\be}\right)^2$.\\

\noindent
Taking for instance $\varepsilon = 1/2$ gives the existence of a constant $\lambda>0$ such that
\[
\liminf\limits_{n\rightarrow\infty}~\frac{\log \p\big(\znb<-0.5m_n\big)}{2^n}\geq -\lambda.
\]

\subsection{Lower bound for the left tail of $\zb$}

We now show that $\znb$ is a good approximation of $\zb$ when estimating the left tail below a fraction of $m_n$.
Fix $\delta>0$, and let us write 
\begin{align*}
\p\Big(\big|&Z_\be-\znb\big|\geq \delta\big|\znb\big|\,\Big|\, \fn\Big)\\
&=\p\left(\bigg| \frac{1}{2^n}\sun \e^{\beta X_u- \frac{\beta^2}{2}n }   \,\zb^u+  \frac{1}{2^n} \sun \e^{\beta X_u- \frac{\beta^2}{2}n } \big(X_u-\beta n\big) \big(\wb^u-1\big)\bigg|\geq \delta|\znb| \,\bigg|\, \fn\right)\\
&\leq \p\bigg(\bigg|  \frac{1}{2^n} \sun \e^{\beta X_u- \frac{\beta^2}{2}n }    \,\zb^u\bigg|\geq \frac{\delta|\znb|}{2} \,\bigg|\, \fn\bigg)\\&+\p\bigg(\bigg| \frac{1}{2^n}  \sun \e^{\beta X_u- \frac{\beta^2}{2}n }   \big(X_u-\beta n\big) \big(\wb^u-1\big)\bigg|\geq \frac{\delta|\znb|}{2} \,\bigg|\, \fn\bigg)\\
&\leq \left(\frac{2}{\delta|\znb|}\right)^p B_p \bigg(\es|\zb|^p  \frac{1}{2^{np}} \sun \e^{p \beta X_u- \frac{p \beta^2}{2}n }   +\es|\wb-1|^p \frac{1}{2^{np}} \sun \e^{p \beta X_u- \frac{p \beta^2}{2}n } |X_u-\be n|^p  \bigg),
\end{align*}
using Lemma \ref{MZ} for the last inequality with some $p\in[1,\gamma\wedge2)$.
On the event $\left\{ (X_u)_{|u|=n}\in\pne \right\}$, we have
\[
|Z_n(\beta)|\geq\frac{1-\ep}{\e\beta}\,\e^{\frac{\beta^2}{2}n},
\]
\[
\frac{1}{2^{np}} \sun  e^{p \beta X_u- \frac{p \beta^2}{2}n }    \leq \e^{-p\am}\frac{1}{2^{(p-1)n}}\,\e^{p\bh n},
\]
\[
\frac{1}{2^{np}} \sun e^{p \beta X_u- \frac{p \beta^2}{2}n }    |X_u-\be n|^p \leq \bigg(\frac{\e^{-\am}\ap}{\be}\bigg)^p\frac{1}{2^{(p-1)n}}\,\e^{p\bh n}.
\]
Therefore,
\begin{align}\label{ineg}
\p\Big(\big|Z_\be-\znb\big|\geq \delta|\znb|\,\Big|\, \pne\Big)&\leq \bigg(\frac{2\e\be\e^{-\am}}{\delta(1-\ep)}\bigg)^p B_p \left( \es[|\zb|^p]+\es[|\wb-1|^p]\left(\frac{\ap}{\be}\right)^p \right)\frac{1}{2^{(p-1)n}}.
\end{align}
And the right term decays exponentially to $0$ as soon as $p>1$. Then, on $\pne\cap\{|Z_\be-\znb|\leq \delta|\znb|\}$, one has

\[
Z_\be\leq-(1-(\delta+\ep))m_n,
\]
thus
\begin{align*}
\p\left(Z_\be\leq-(1-(\delta+\ep))m_n\right)&\geq\p((X_u)\in\pne)\,\p\left(|Z_\be-\znb|\leq \delta|\znb|\,\big|\,(X_u)\in\pne\right)\\
&\geq \e^{-\kappa_\ep2^n+O(n^2)}\left(1-\frac{C(\ep,\delta)}{2^{(p-1)n}}\right),
\end{align*}
where $C(\ep,\delta)$ is the term in the r.h.s of Equation (\ref{ineg}).
This last inequality and the continuity of $\kappa_\ep$ for $\ep\in(0,1)$ lead to
\[
\liminf\limits_{n\to\infty} \frac{\log\p(\zb\leq-(1-\ep)m_n)}{2^n}\geq -\kappa_\ep.
\]
And therefore,
\[
\liminf\limits_{x\to\infty} \frac{\log\p(\zb\leq-x)}{x^\gamma}> -\infty
\]
which concludes the proof of Theorem \ref{lowerTh}.


\section{Upper bound for the left tail}


\subsection{Large deviations with the branching property}

From now on, we omit the subscript $\beta$ to alleviate notations. The branching property gives the decomposition
\begin{align}\label{brch}
Z &=  \frac{1}{2^{n}} \sun  \e^{ \beta X_u- \frac{ \beta^2}{2}n }   \big(Z^u +(X_u-\beta n)W^u\big).
\end{align}

In order to obtain large deviations bounds, we need to control the Laplace transform of the random variables $Z+aW$. Intuitively, for $a>0$, the fat right tail of $W$ should help. For the same reason, it seems hopeless at first sight to obtain a nice bound in the case $a<0$ . But there must be some compensation with $Z$ because a typical scenario for a large value for $W$ leads to a positive large value for $Z$ too (the box scenario of Subsection \ref{boxsce} being very atypical). Let us rewrite things a bit:
\begin{align*}
	Z_n+aW_n &= \frac{1}{2^{n}} \sun \e^{ \beta X_u- \frac{ \beta^2}{2}n }   \left(X_u-\beta n+a\right)\\
	&= \e^{-\beta a} \frac{1}{2^{n}}  \sun \e^{\beta (X_u+a)- \frac{ \beta^2}{2}n}((X_u+a)-\beta n).
\end{align*}
Therefore $Z+aW = \e^{-\beta a} Z^{[a]}$ where $Z^{[a]}$ is the limit of the derivative martingale when the initial ancestor starts at $a$. Since the large negative values for $Z_n$ are obtained when all the particles lie around $\beta n$ at time $n$, it seems reasonable to expect that starting from $a<0$ makes it harder to achieve. In fact, by exponential tilting, we can shift the mean of the $X_u$'s and obtain a bound on the Laplace transform of $Z^{[a]}$ with the Laplace transform of $Z$.
\begin{lem}\label{ZW}
	For $a\in\mathbb R$ and $\lambda>0$, ones has
	\[
	\es\left[\e^{-\lambda(Z+aW)}\right]\leq \e^{\frac{a^2}{2}} \,\es\left[\e^{-2\lambda\e^{-\beta a} Z}\right]^{\frac{1}{2}}.
	\]
\end{lem}
\begin{proof}
	Recall that $(X_u)_{|u|=n}$ is a centered Gaussian vector with covariance matrix $\Sigma_n$ and note that $\frac{1}{2^n}\sun X_u \sim\mathcal N(0,1-2^{-n})$. Under the probability $\mathbb Q_a$ defined by
	\[
	\dfrac{\mathrm d \mathbb Q_a}{\mathrm d \p} = Y_n^{[a]}\de \e^{\frac{a}{(1-2^{-n})2^n}\sun X_u-\frac{a^2}{2(1-2^{-n})}},
	\]
	the variables $(X_u)_{|u|=n}$ are shifted: we have $\es_{\mathbb Q_a}[X_u]=a$ and the covariance matrix remains unchanged. Therefore,
	\begin{align*}
		\es\left[\e^{-\lambda Z_n^{[a]}}\right]&=\es\left[\e^{-\lambda  \frac{1}{2^n}\sun\e^{\beta(X_u+a)-\frac{\beta^2}{2}n}(X_u+a-\beta n)}\right]\\
		&=\es\left[Y_n^{[a]} \,\e^{-\lambda \frac{1}{2^n} \sun \e^{\beta X_u-\frac{\beta^2}{2}n}(X_u-\beta n)}\right]\\
		&=\e^{-\frac{a^2}{2(1-2^{-n})}}\,\es\left[\e^{\frac{a}{(1-2^{-n})2^n}\sun X_u}\e^{-\lambda \frac{1}{2^n} \sun \e^{\beta X_u-\frac{\beta^2}{2}n}(X_u-\beta n)   }\right]\\
		&\leq\e^{-\frac{a^2}{2(1-2^{-n})}}\,\es\left[\e^{\frac{2a}{(1-2^{-n})2^n}\sun X_u}\right]^{\frac{1}{2}}\,\es\left[\e^{-2\lambda  \frac{1}{2^n} \sun \e^{\beta X_u-\frac{\beta^2}{2}n}(X_u-\beta n) }\right]^{\frac{1}{2}}\\
		&=\e^{\frac{a^2}{2(1-2^{-n})}}\,\es\left[\e^{-2\lambda Z_n}\right]^{\frac{1}{2}}.
	\end{align*}
	Then, applying Fatou's lemma and conditional Jensen's inequality yields
	\begin{align*}
		\es\left[\e^{-\lambda Z^{[a]}}\right]\leq\liminf\limits_{n\to\infty}\e^{\frac{a^2}{2}}\es\left[\e^{-2\lambda\es\left[Z|\mathcal F_n\right]}\right]^\frac{1}{2}\leq\e^{\frac{a^2}{2}}\es\left[\e^{-2\lambda Z}\right]^\frac{1}{2}.
	\end{align*}
	Finally, using $Z+aW = \e^{-\beta a} Z^{[a]}$ gives
	\[
	\es\left[\e^{-\lambda(Z+aW)}\right]=\es\left[\e^{-\lambda \e^{-\beta a} Z^{[a]}}\right]\leq \e^{\frac{a^2}{2}}\,\es\left[\e^{-2\lambda\e^{-\beta a} Z}\right]^{\frac{1}{2}}.
	\]
\end{proof}

For high particles and under the sub-Gaussian assumption\footnote{Here we think of sub-Gaussian on the left, meaning that a random variable $V$ is sub-Gaussian if there exist $c$ and $C>0$ such that for $\lambda\geq0$, we have $\es\e^{-\lambda V}\leq C\e^{c\lambda^2}$. This will be proved in Proposition \ref{subG}.}, we have the more refined lemma:
\begin{lem}\label{highlap}
	For $a>0$ and $\lambda\geq0$, we have
	\[
	\es\left[\e^{-\lambda (Z+aW)}\right]\leq
	\e^{C(\lambda\e^{-\beta a} a+(\lambda \e^{-\beta a})^2)}.
	\]
\end{lem}
\begin{proof}
With the sub-Gaussian result of Proposition \ref{subG}, we can sharpen the previous lemma using H\"older's inequality instead of Cauchy-Schwarz in the proof of Lemma \ref{ZW} . If $a>0$, one has
\begin{align*}
	\es\left[\e^{-\lambda Z^{[a]}}\right]&\leq \e^{(p-1)\frac{ a^2}{2}}\,\es\left[\e^{-\frac{p}{p-1}\lambda Z}\right]^{\frac{p-1}{p}}\\
	&\leq
	\e^{(p-1)\frac{ a^2}{2}+C\frac{p}{p-1}\lambda^2}.
\end{align*}
By choosing $p = 1+\frac{\lambda}{a}$, one gets
\[
\es\left[\e^{-\lambda Z^{[a]}}\right]\leq
\e^{C(\lambda a+\lambda^2)},
\]
for some other constant $C>0$. Therefore
\begin{equation*}\label{ZW2}
	\es\left[\e^{-\lambda (Z+aW)}\right]\leq
	\e^{C(\lambda\e^{-\beta a} a+(\lambda \e^{-\beta a})^2)}.\\
\end{equation*}
\end{proof}

\noindent
We now turn to the proof of Theorem \ref{lowerThbis}.

\vspace{0.2 cm}

\noindent
\emph{Proof of Theorem \ref{lowerThbis}} :
	We start with representation (\ref{brch}) and apply a Chernoff bound to the three following terms
	\begin{align*}
		Z &=  \frac{1}{2^{n}} \sum_{X_u<\beta n}   \e^{\beta X_u-\frac{\beta^2}{2}n} \big(Z^u +(X_u-\beta n)W^u\big)\\
		&+  \frac{1}{2^{n}} \sum_{X_u\in[\beta n,\beta n+1]}  \e^{\beta X_u-\frac{\beta^2}{2}n} \big(Z^u +(X_u-\beta n)W^u\big)\\&+  \frac{1}{2^{n}} \sum_{X_u>\beta n+1}  \e^{\beta X_u-\frac{\beta^2}{2}n} \big(Z^u +(X_u-\beta n)W^u\big).
	\end{align*}
From now on, fix some $\alpha>0$ whose value will be determined later.\\

For the particles below $\beta n$, we have, for $\lambda>0$,
\begin{align*}
	\p\Bigg( \frac{1}{2^{n}} \sum_{X_u<\beta n}&  \e^{\beta X_u-\frac{\beta^2}{2}n} \big(Z^u +(X_u-\beta n)W^u\big)<-\alpha\e^{\frac{\beta^2}{2}n}\,\bigg|\, \mathcal F_n\Bigg)  \\
	&=\es\left[\exp\left(-\lambda \frac{1}{2^{n}} \sum_{X_u<\beta n}  \e^{\beta X_u-\frac{\beta^2}{2}n}  \big(Z^u +(X_u-\beta n)W^u\big)\right)>\exp\left(\lambda\alpha\e^{\frac{\beta^2}{2}n}\right)\,\Bigg|\, \mathcal F_n\right] \\
	&\leq \exp\left(-\lambda \alpha \e^{\frac{\beta^2}{2}n}\right) \prod_{X_u<\beta n}\es\left[ \exp\left(-\lambda   \frac{1}{2^{n}} \, \e^{\beta X_u-\frac{\beta^2}{2}n}  \big(Z^u +(X_u-\beta n)W^u\big)\right)\,\bigg|\, \mathcal F_n\right]. \\
\end{align*}
Then, using Proposition \ref{lowlap},
\begin{align*}
	\es\Bigg[ \exp\bigg(-\lambda  \frac{1}{2^{n}}\,  &\e^{\beta X_u-\frac{\beta^2}{2}n}  \big(Z^u +(X_u-\beta n)W^u\big)\bigg)\,\bigg|\, \mathcal F_n\Bigg]\\
	&\leq \exp\left(-\lambda  \,\frac{1}{2^{n}}  \e^{\beta X_u-\frac{\beta^2}{2}n}   (X_u-\beta n)+C\left(\left(\lambda  \frac{1}{2^{n}}  \e^{\beta X_u-\frac{\beta^2}{2}n} \e^{-\beta(X_u-\beta n)}\right)^2+1\right)\right)\\
	&\leq \exp\left(\frac{1}{\e\beta}\lambda\frac{\e^{\frac{\beta^2}{2}n}}{2^n}+C\left(\left(\lambda\frac{\e^{\frac{\beta^2}{2}n}}{2^n}\right)^2+1\right)\right).
\end{align*}
Therefore, by choosing $\lambda=2^n \e^{-\frac{\beta^2}{2}n}$, one gets
\begin{align*}
	\p\Bigg(\sum_{X_u<\beta n}  \frac{1}{2^{n}}\,  \e^{\beta X_u}&^{-\frac{\beta^2}{2}n} \big(Z^u +(X_u-\beta n)W^u\big)<-\alpha\e^{\frac{\beta^2}{2}n}\,\bigg|\, \mathcal F_n\Bigg)&\hspace{5cm}\\
	&\leq \exp\left(- \alpha 2^n\right) \prod_{X_u<\beta n}\exp\left(\frac{1}{\e\beta}+2C\right)\\
	&\leq \exp\left(- \Big(\alpha-\frac{1}{\e\beta}-C\Big) 2^n\right).\\
\end{align*}

For particles between $\beta n$ and $\beta n+1$, Lemma \ref{ZW} and Proposition \ref{subG} provide $C>0$ such that, for $\lambda,a\geq0$,
\[
\es\left[\e^{-\lambda(Z+aW)}\right]\leq \e^{\frac{a^2}{2}+C\left(\lambda e^{-\beta a}\right)^2}.
\]
Then,
\begin{align*}
	\p\Bigg(  \frac{1}{2^{n}}  &\sum_{\beta n\leq X_u\leq\beta n+1}  \e^{\beta X_u-\frac{\beta^2}{2}n} \big(Z^u +(X_u-\beta n)W^u\big)<-\alpha\e^{\frac{\beta^2}{2}n}\,\bigg|\, \mathcal F_n\Bigg)  \\
	&\leq \exp\left(-\lambda \alpha \e^{\frac{\beta^2}{2}n}\right) \prod_{\beta n\leq X_u\leq\beta n+1}\es\left[ \exp\left(-\lambda  \frac{1}{2^{n}}  \,\e^{\beta X_u-\frac{\beta^2}{2}n} \big(Z^u +(X_u-\beta n)W^u\big)\right)\,\bigg|\, \mathcal F_n\right] \\
	&\leq \exp\left(-\lambda \alpha \e^{\frac{\beta^2}{2}n}\right) \prod_{\beta n\leq X_u\leq\beta n+1} \exp\left(\frac{1}{2} (X_u-\beta n)^2+\left(\lambda  \frac{1}{2^{n}}  \,\e^{\beta X_u-\frac{\beta^2}{2}n} \e^{-\beta(X_u-\beta n)}\right)^2\right) \\
	&\leq \exp\left(-\lambda \alpha \e^{\frac{\beta^2}{2}n}\right) \prod_{\beta n\leq X_u\leq\beta n+1} \exp\left(\frac{1}{2}+\left(\lambda \,\frac{\e^{\frac{\beta^2}{2}n}}{2^n}\right)^2\right).
\end{align*}
Taking again $\lambda = 2^n\,\e^{-\frac{\beta^2}{2}n}$ gives
\[
	\p\Bigg(\frac{1}{2^n}\sum_{\beta n\leq X_u\leq\beta n+1} \e^{\beta X_u - \frac{\beta^2}{2}n} \big(Z^u +(X_u-\beta n)W^u\big)<-\alpha\e^{\frac{\beta^2}{2}n}\,\bigg|\, \mathcal F_n\Bigg)  \leq \exp\left(- \left(\alpha-\frac{3}{2}\right) 2^n\right).\\
\]

It remains to deal with the particles above $\beta n+1$. An estimate on the Laplace transform of $W$ is needed for this purpose. Recall that
\[
W = \frac{1}{2}\,\e^{\beta X_0-\frac{\beta^2}{2}}\,W^0+ \frac{1}{2} \,\e^{\beta X_1-\frac{\beta^2}{2}}\,W^1,
\]
where $W^0$ and $W^1$ are independent and have the same law as $W$.
A straightforward computation shows that
\[
\lim\limits_{x\rightarrow 0}\frac{\log\log1/\p(  \frac{1}{2}\,\e^{\beta X_0-\frac{\beta^2}{2}} <x)}{\log\log1/x}=2.
\]
And \cite[Theorem 1.2]{nikula20} provides $c>0$ such that for every $\lambda\geq1$, one has\footnote{This is true for any exponent strictly below 2, here we choose 3/2.}
\begin{equation}\label{Wlap}
	\es\left[\e^{-\lambda W}\right]\leq \e^{-c\log^{3/2}(\lambda)}.\\
\end{equation}
When $x>\beta n+1$, we can use a portion of $W$ to improve the Chernoff bound:
\begin{align*}
	\es&\left[\exp\left(-\lambda  \frac{1}{2^n}\e^{\beta x-\frac{\beta^2}{2}n}\big(Z +(x-\beta n)W\big)\right)\right]\\
	&=\es\left[\exp\left(-\lambda \frac{1}{2^n}\e^{\beta x-\frac{\beta^2}{2}n} \big(Z +(x-\beta n-1 )W\big)\right)\exp\left(-\lambda \frac{1}{2^n}\e^{\beta x-\frac{\beta^2}{2}n} \,W\right)\right]\\
	&\leq \es\left[\exp\left(-2\lambda \frac{1}{2^n}\e^{\beta x-\frac{\beta^2}{2}n} \big(Z +(x-\beta n-1 )W\big)\right)\right]^{\frac{1}{2}}\es\left[\exp\left(-2\lambda \frac{1}{2^n}\e^{\beta x-\frac{\beta^2}{2}n} \,W\right)\right]^{\frac{1}{2}}.
\end{align*}
Then using Lemma \ref{highlap} for the first term and Equation (\ref{Wlap}) for the second term yields
\begin{align*}
	\es\Big[\exp&\Big(-\lambda \frac{1}{2^n}\e^{\beta x-\frac{\beta^2}{2}n} \big(Z +(x-\beta n)W\big)\Big)\Big]\\
	&\leq \exp\Bigg(C\lambda \frac{1}{2^n}\e^{\beta x-\frac{\beta^2}{2}n} \e^{-\beta(x-\beta n-1)}(x-\beta n-1 )+2C\Big(\lambda \frac{1}{2^n}\e^{\beta x-\frac{\beta^2}{2}n} \e^{-\beta(x-\beta n-1)}\Big)^2\\
	&\hspace{9cm}-\frac{c}{2}\log^{3/2}\Big(2\lambda \frac{1}{2^n}\e^{\beta x-\frac{\beta^2}{2}n} \Big)\Bigg)\\
	&= \exp\left(C\lambda\frac{\e^{\frac{\beta^2}{2}n}}{2^n}\,\e^{\beta}(x-\beta n-1 )+2C\Big(\lambda\frac{\e^{\frac{\beta^2}{2}n}}{2^n}\e^{\beta}\Big)^2-\frac{c}{2}\log^{3/2}\Big(2\lambda \frac{1}{2^n}\e^{\beta x-\frac{\beta^2}{2}n} \Big)\right),
\end{align*}
as soon as $2\lambda\, \frac{1}{2^n}\,\e^{\beta (\beta n+1)-n\frac{\beta^2}{2}} \geq 1$. For $\lambda =\e^{-\beta} \,2^n\,\e^{-\frac{\beta^2}{2}n}$, this condition is fulfilled and we obtain
\begin{align*}
	\es\left[\exp\left(-\lambda\frac{1}{2^n}\e^{\beta x-\frac{\beta^2}{2}n}\big(Z +(x-\beta n)W\big)\right)\right]
	&\leq \exp\left(C(x-\beta n-1 )+2C-\frac{c\beta^{3/2}}{2}\big(x-\beta n-1 \big)^{3/2}\right)\\
	&\leq\exp\left(D\right),
\end{align*}
for some $D>0$ \footnote{Here we use the fact that $\beta>0$.}. Choosing $\lambda =\e^{-\beta} \,2^n\e^{-\frac{\beta^2}{2}n}$ gives
\begin{align*}
	\p\bigg(\frac{1}{2^n} \sum_{X_u>\beta n+1} & \e^{\beta x-\frac{\beta^2}{2}n} \big(Z^u +(X_u-\beta n)W^u\big)<-\alpha\e^{\frac{\beta^2}{2}n}\,\bigg|\,\mathcal F_n\bigg)\\
	&\leq \exp\left(-\lambda \alpha\e^{\frac{\beta^2}{2}n}\right) \prod_{X_u>\beta n+1}\es\left[ \exp\left(-\lambda \frac{1}{2^n}\e^{\beta X_u-n\frac{\beta^2}{2}} \big(Z^u +(X_u-\beta n)W^u\big)\right)\,\bigg|\, \mathcal F_n\right] \\
	& \leq \exp\left(-(\e^{-\beta} \alpha-D)\,2^n\right).
\end{align*}
With $\alpha=3\max\left(1+C+\frac{1}{\e\beta},\frac{5}{2},(1+D)\,\e^\beta\right)$, one gets
\begin{align*}
	\p\left(Z<-\alpha\e^{\frac{\beta^2}{2}n}\right) &\leq\p\left(\frac{1}{2^n} \sum_{X_u<\beta n} \e^{\beta X_u-\frac{\beta^2}{2}n} \big(Z^u +(X_u-\beta n)W^u\big)<-\frac{\alpha}{3}\,\e^{\frac{\beta^2}{2}n}\right)\\
	&+\p\left( \frac{1}{2^n} \sum_{\beta n\leq X_u\leq\beta n+1} \e^{\beta X_u-\frac{\beta^2}{2}n} \big(Z^u +(X_u-\beta n)W^u\big)<-\frac{\alpha}{3}\,\e^{\frac{\beta^2}{2}n}\right)\\
	&+\p\left( \frac{1}{2^n} \sum_{X_u>\beta n+1}  \e^{\beta X_u-\frac{\beta^2}{2}n} \big(Z^u +(X_u-\beta n)W^u\big)<-\frac{\alpha}{3}\,\e^{\frac{\beta^2}{2}n}\right)\\
	&\leq 3\exp(-2^n).
\end{align*}
Now, if $x>1$, picking $n\in\mathbb N$ such that $\alpha\e^{\frac{\beta^2}{2}n}\leq x<\alpha\e^{\frac{\beta^2}{2}(n+1)}$ provides a $c>0$ such that
\[
\p\left(Z<-x\right)\leq 3\e^{-cx^\gamma}
\]
and concludes the proof.

\subsection{A continuous analogue}
The aim of this section is twofold. First we prove that $Z$ displays a sub-Gaussian left tail in the $L^4$ phase  ($\beta<\beta_c/2$) using a related model. Then, we use the same techniques to recover a uniform bound on the Laplace transform of $Z+aW$ when $a<0$.\\

We consider the binary branching Wiener process which is defined the following way:  start with one particule at $0$ that splits into two particles which diffuse as standard Brownian motions. At time $1$, those two particles split into two new particles which diffuse independently from the position of their ancestor and so on. More precisely, if $(B^v)_{v\in\mathbb T}$ is a family of i.i.d. Brownian motions and $\nt \de \{0,1\}^{\lceil t \rceil}$, where $\lceil t \rceil$ is the smallest integer above $t$, define for $t\geq0$ and $u\in\mathcal N_t$
\[
X_t(u)\de\sum_{\varnothing <v\leq u} B^v_1+B^u_{t-\lfloor t \rfloor},
\]
where $\lfloor t \rfloor$ is the integer part of $t$.
At time $t$, we thus have $2^{\lceil t \rceil}$ particles and at integer times $n\in\mathbb N$ the process is distributed as the original branching random walk. This detour by the continuous case will allow us to perform stochastic calculus.\\

The analogue of the derivative martingale is given by
\[
Z_t \de \frac{1}{2^{\lceil t \rceil}}\summ \left(X_t(u)-\beta t\right)\ebt.
\]
Note that despite the integer part in the definition of $Z_t$, it is a continuous martingale with respect to $\mathcal F_t\de \sigma(X_s(u), u\in\mathcal N_s,s\leq t)$. Its quadratic variation is given by
\[
\braket{Z}_t = \frac{1}{4^{\lceil t \rceil}}\sum_{u,v\in\nt}\int_{0}^{t} \left(1+\beta(X_s(u)-\beta s)\right)\left(1+\beta(X_s(v)-\beta s)\right)\ebsu \ebsv \,\ind_{s\leq u\wedge v}\,\ds,
\]
where $u\wedge v$ is the generation of the last common ancestor of $u$ and $v$.\\

Let $\ft\de\{0,1\}^{\mathbb N}$ and for $u\in\ft$, let $X_t(u)$ denotes the position of the ancestor of $u$ at time $t$. $\ft$ is endowed with the ultrametric distance $\mathrm d$ defined by $\mathrm d(u,v)=2^{-u\wedge v}$ and the uniform probability measure $\mu$ which is characterized by $\mu(\{v : v\geq u\} )= 2^{-|u|}$. The quadratic variation can be reformulated as
\begin{align*}
\braket{Z}_t &=\int_{\ft^2\times [0,t]} \left(1+\beta(X_s(u)-\beta s)\right)\left(1+\beta(X_s(v)-\beta s)\right)\ebuv \,\ind_{s\leq u\wedge v}\,\mu(\mathrm du)\mu(\mathrm dv)\,\ds\\
&=\int_{\ft^2\times [0,t]} \left(1+\beta(X_s(u)-\beta s)\right)^2 \e^{2\beta X_s(u)-\beta^2 s} \,\ind_{s\leq u\wedge v}\,\mu(\mathrm du)\mu(\mathrm dv)\,\ds.
\end{align*}

The approach laid in \cite{lacoinrhodesvargas22} is to define an auxiliary martingale $\tilde{Z}$ where the excursions above some well chosen threshold are removed and to show Gaussian concentration for this new martingale and the difference with the original one by proving that their brackets remain bounded. We are reproducing the main ideas of the proof here since we are going to use them for the proof of Proposition \ref{lowlap}. 
The following proposition states that the left tail of $Z$ is sub-Gaussian in the $L^4$ phase.

\begin{prop}\label{subG}
	Let $\beta<\frac{\beta_c}{2}$, then there exists $C>0$ such that, for $\lambda\geq0$,
	\[
	\es\left[\e^{-\lambda Z}\right]\leq \e^{C\lambda^2}.
	\]
\end{prop}
\begin{proof}
Let us define the following stopping times:
\begin{align*}
	T_k^u&\de \inf\left\{t\geq R_{k-1}^u \, : \, X_t(u)=(\beta+\eta)t+A\right\},\\
	R_k^u&\de \inf\left\{t\geq T_{k}^u \, : \, X_t(u)=\beta t\right\},
\end{align*}
where $R_0^u\de 0$. Let $\mathcal{R}^u\de \cup_k [ R_{k-1}^u,T_k^u ]$ and define

\begin{align*}
	\tilde{Z}_t &\de\int_{\ft} \int_0^t \left(1+\beta(X_s(u)-\beta s)\right)\ebsu\,\ind_{s\in\mathcal R^u} \,\mathrm dX_s(u) \, \mu(\mathrm du).
\end{align*}
It is a martingale and the difference with the original martingale $Z$ is controlled by
\begin{align*}
	Q &= \int_{\ft} \sum_{k\geq 1} \ind_{\{T_k^u<+\infty\}}\big(A+\eta T_k^u\big)\e^{\big(\frac{\beta^2}{2}+\beta\eta\big)T_k^u+\beta A}\,\mu(\mathrm du) \\
	& \eqqcolon\int_{\ft} Q^u \mu(\mathrm du),
\end{align*}
and the fact that $Z_\infty\geq\tilde Z_\infty- Q$, see \cite[Section 5.2]{lacoinrhodesvargas22} for more details.\\

Now we prove that both right terms exhibit Gaussian concentration.
The bracket of the first one is given by
\begin{align*}
	\braket{\tilde{Z}}_t &=\int_{\ft^2\times [0,t]} \left(1+\beta(X_s(u)-\beta s)\right)\left(1+\beta(X_s(v)-\beta s)\right)\ebuv \\
	&\hspace{9cm} \ind_{\{s\in\mathcal R^u \cap \mathcal R^v\}}\,\ind_{\{s\leq u\wedge v\}}\,\mu(\mathrm du)\mu(\mathrm dv)\,\ds\\
	\\
	&=\int_{\ft^2\times [0,t]} \left(1+\beta(X_s(u)-\beta s)\right)^2 \e^{2\beta X_s(u)-\beta^2 s} \,\ind_{\{s\in\mathcal R^u \cap \mathcal R^v\}}\,\ind_{\{s\leq u\wedge v\}}\,\mu(\mathrm du)\mu(\mathrm dv)\,\ds.
\end{align*}

Using the fact that $X_s(u)\leq (\beta+\eta)s+A$ when $s\in\mathcal R^u$, one gets
\begin{align*}
	\braket{\tilde{Z}}_t
	&\leq\int_{\ft^2\times [0,t]} \left(1+\beta(\eta s+A)\right)^2 \e^{(\beta^2+2\beta\eta)s+2\beta A } \,\ind_{\{s\leq u\wedge v\}}\,\mu(\mathrm du)\mu(\mathrm dv)\,\ds.
\end{align*}
Now, note that the following integral
\[
\int_{\ft^2} \e^{\alpha u\wedge v}\,\mu(\mathrm du)\mu(\mathrm dv) = \sum_{n\geq0} \frac{\e^{\alpha n}}{2^{n+1}}
\]
is finite if and only if $\alpha<\log2$. This guaranties that $\braket{\tilde{Z}}_\infty$ is bounded if $\beta^2+2\beta\eta<\log2$ and provides a sub-Gaussian tail for $\tilde{Z}_\infty$.\\

To prove Gaussian concentration for $Q$, the authors in \cite{lacoinrhodesvargas22} start by showing that $\es Q<\infty$. Then they use the continuous martingale $Q^u_t\de\es\left[Q^u\,|\,\mathcal F_t\right]$ and the decomposition
\[
\mathrm dQ^u_t\de A^u_t\,\mathrm dX_t(u).
\]
This way, the bracket is given by
\[
\braket{Q}_\infty = \int_{\ft^2\times [0,\infty)} A^u_t A^v_t \,\ind_{\{t\leq u\wedge v\}}\, \,\mu(\mathrm du) \mu(\mathrm dv) \dt.
\]
The method used in \cite{lacoinrhodesvargas22} to obtain the expression of $A^u_t$ is \guillemotleft purely Brownian\guillemotright ~and does not involve the covariance structure, it is therefore valid in our context. Let us reproduce the results here.\\

When $t\in(T_k^{u},R_k^{u})$, the Markov property for $(X_t(u))_{t\geq0}$ gives 
\[
Q^u = \sum_{i=1}^{k} (A+\eta T_i^{u})\e^{\big(\frac{\beta^2}{2}+\beta\eta\big)T_i^{u}}+\es_{X_t(u)}\left[\sum_{i\geq 1} \ind_{\widehat{T}_i^{t}<+\infty}\big(A+\eta (\widehat{T}_i^{t}+t)\big)\e^{\big(\frac{\beta^2}{2}+\beta\eta\big)(\widehat{T}_i^{t}+t)+\beta A}\right],
\]
where $\es_z$ denotes the expectation with respect to the law of a standard Brownian motion starting at $z$ and $\widehat{T}_0^t\de0$,
\begin{align*}
	\widehat{R}_k^t&\de\inf\left\{s\geq\widehat{T}^t_{k-1}\,:\,B_s\leq\beta(t+s)\right\},\\
	\widehat{T}_k^t&\de\inf\left\{s\geq\widehat{R}^t_{k}\,:\,B_s=A+(\beta+\eta)(t+s)\right\}.
\end{align*}
In this case, this yields to the following
\[
A^u_t =\left. \partial_z \left(\es_{z}\left[\sum_{i\geq 1} \ind_{\widehat{T}_i^{t}<+\infty}\big(A+\eta (\widehat{T}_i^{t}+t)\big)\e^{\big(\frac{\beta^2}{2}+\beta\eta\big)(\widehat{T}_i^{t}+t)+\beta A}\right]\right) \right\rvert_{z=X_t(u)}.
\]
\noindent
When $t\in(R_k^{u},T_{k+1}^{u})$, in the same manner,
\[
Q^u = \sum_{i=1}^{k} (A+\eta T_i^{u})\e^{\big(\frac{\beta^2}{2}+\beta\eta\big)T_i^{u}}+\es_{X_t(u)}\left[\sum_{i\geq 1} \ind_{{T}_i^{t}<+\infty}\big(A+\eta ({T}_i^{t}+t)\big)\e^{\big(\frac{\beta^2}{2}+\beta\eta\big)({T}_i^{t}+t)+\beta A}\right],
\]
where ${R}_0^t\de0$ and
\begin{align*}
	{T}_k^t&\de\inf\left\{s\geq{R}^t_{k-1}\,:\,B_s\leq A+(\beta+\eta)(t+s)\right\},\\
	{R}_k^t&\de\inf\left\{s\geq{R}^t_{k}\,:\,B_s=\beta(t+s)\right\}.
\end{align*}
Then
\[
A^u_t = \left.\partial_z \left(\es_{z}\left[\sum_{i\geq 1} \ind_{{T}_i^{t}<+\infty}\big(A+\eta ({T}_i^{t}+t)\big)\e^{\big(\frac{\beta^2}{2}+\beta\eta\big)({T}_i^{t}+t)+\beta A}\right]\right)\right\rvert_{z=X_t(u)}.
\]
In a more compact form, this gives
\[
	A^u_t = \left\{
	\begin{array}{ll}
		f_1(t,X_t(u)), & \mbox{ if } t\in(T_k^u,R_k^u), \\
		f_2(t,X_t(u)), & \mbox{ if } t\in(R_k^u,T_{k+1}^u).
	\end{array}
	\right.
\]
If one chooses $\eta$ such that $2\eta>\beta$ and $\beta^2+2\beta\eta<\log2$ (which is possible in the $L^4$ phase), \cite[Lemma 5.3, Lemma 5.4]{lacoinrhodesvargas22}\footnote{You should read $z\leq(\gamma+\eta)t+A$ in the statement of  Lemma 5.4.} provides
$C=C(A,\eta,\beta)>0$ such that
\begin{align*}
	|f_1(t,z)|&\leq C(t+1)\e^{\big(\frac{\beta^2}{2}+\beta\eta\big)t},\quad \text{ for } z\geq \beta t,\\
	|f_2(t,z)|&\leq C(t+1)\e^{\big(\frac{\beta^2}{2}+\beta\eta\big)t},\quad \text{ for } z\leq (\beta+\eta) t+A.\\
\end{align*}
Thus,
\[
\braket{Q}_\infty \leq \int_{\ft^2\times [0,\infty)} C^2(t+1)^2\e^{(\beta^2+2\beta\eta)t} \,\ind_{t\leq u\wedge v}\, \,\mu(\mathrm du) \mu(\mathrm dv) \dt<\infty,
\]
which proves that $Q$ displays Gaussian concentration.
\end{proof}
The above approach also provides a useful bound on the Laplace transform of low particles.

\begin{prop}\label{lowlap}
There exists $C>0$ such that, for $a\leq0$ and $\lambda\geq0$,
\[
\es\left[\e^{-\lambda(Z+aW)}\right]\leq \e^{-\lambda a+C\left(\left(\lambda \e^{-\beta a}\right)^2+1\right)}.
\]
\end{prop}
\begin{proof}
As in the discrete case, define the derivative martingale when the initial ancestor starts at $a\leq0$:
\[
Z_t^{[a]} \de \frac{1}{2^{\lceil t \rceil}}\summ \left(a+X_t(u)-\beta t\right)\ebta.
\] 
Define as previously
\begin{align*}
	T_k^{u,[a]}&\de \inf\left\{t\geq R_{k-1}^u \, : \, a+X_t(u)=(\beta+\eta)t+A\right\},\\
	R_k^{u,[a]}&\de \inf\left\{t\geq T_{k}^u \, : \, a+X_t(u)=\beta t\right\},
\end{align*}
where $R_0^{u,[a]}\de 0$. Let $\mathcal{R}^{u,[a]}\de \bigcup_k [ R_{k-1}^{u,[a]},T_k^{u,[a]} ]$ and define
\begin{align*}
	\tilde{Z}^{[a]}_t &=a\e^{\beta a}+\int_{\ft} \int_0^t \left(1+\beta(a+X_s(u)-\beta s)\right)\ebsa\,\ind_{s\in\mathcal R^{u,[a]}} \,\mathrm dX_s(u) \, \mu(\mathrm du).
\end{align*}
It is a martingale bounded in $L^2$ by Equation (\ref{bounds}) and we have
\[
Z_\infty^{[a]}\geq \tilde{Z}_\infty^{[a]}-Q^{[a]},
\]
where 
\begin{align*}
	Q^{[a]} &= \int_{\ft}  Q^{[a]}_u\,\mu(\mathrm du)\\
	&=\int_{\ft} \sum_{k\geq 1} \ind_{T_k^{u,[a]}<+\infty}\Big(A+\eta T_k^{u,[a]}\Big)\e^{\big(\frac{\beta^2}{2}+\beta\eta\big)T_k^{u,[a]}+\beta A}\,\mu(\mathrm du).
\end{align*}
We are going to prove
	\begin{align}\label{bounds}
	\sup_{a\leq0} \,\mathrm{ess~sup} \braket{\tilde{Z}^{[a]}}_\infty<\infty,&&\es Q^{[a]}\leq\es Q ,&&\sup_{a\leq0} \,\mathrm{ess~sup} \braket{Q^{[a]}}_\infty<\infty.
	\end{align}
Let us admit for a moment those bounds. Then, using $ \es \tilde Z_\infty^{[a]}=a\e^{\beta a}$,
\begin{align*}
	\es\left[\e^{-\lambda Z_\infty^{[a]}}\right]&\leq \es\left[\e^{-2\lambda \tilde Z_\infty^{[a]}}\right]^\frac{1}{2}
	\es\left[\e^{2\lambda  Q^{[a]}}\right]^\frac{1}{2}\\
	&\leq \e^{-\lambda \es \tilde Z_\infty^{[a]}+\text{ess sup}\braket{\tilde{Z}^{[a]}}_\infty\lambda^2}\,\e^{\lambda\es Q^{[a]}+\text{ess sup}\braket{Q^{[a]}}_\infty\lambda^2}\\
	&\leq \e^{-\lambda a\e^{\beta a}+C(\lambda^2+1)},
\end{align*}
for some $C>0$. And using $Z+aW =\e^{-\beta a}Z^{[a]}$ concludes the proof.\\

Now let us prove the $3$ statements in (\ref{bounds}). First, using the fact that $a+X_s(u)\leq (\beta+\eta)s+A$ when $s\in\mathcal R^{u,[a]}$, one obtains the same finite bound as in the case $a=0$, namely
\begin{align*}
	\braket{\tilde{Z}^{[a]}}_\infty
	&\leq\int_{\ft^2\times [0,\infty)} \left(1+\beta(\eta s+A)\right)^2 \e^{(\beta^2+2\beta\eta)s+2\beta A } \,\ind_{s\leq u\wedge v}\,\mu(\mathrm du)\mu(\mathrm dv)\,\ds<\infty.
\end{align*}
\begin{rem}
	In the case $a>0$, we lose the control on the bracket for $t\in(0,T_1)$.
\end{rem}

The strong Markov property for $(X_t(u),t\geq0)$ applied at $T_1^{u,[a]}$ shows that the conditional law of $\big(T_k^{u,[a]}\big)_{k\geq2}$ given $T_1^{u,[a]}$ is equal to the conditional law of $\big(T_k^{u}\big)_{k\geq2}$ given $T_1^{u}$. We can thus focus on the first hitting time. Now, recall that the hitting time $T_{\alpha,b}$ of a line $s\mapsto \alpha+bs$ with $b>0$ by a standard Brownian motion has a law given by
\[
\ind_{t>0}\,\frac{\alpha}{\sqrt{2\pi}t^{\frac{3}{2}}}\e^{-\alpha b}\e^{-\frac{b^2}{2}t-\frac{\alpha^2}{2t}}+\left(1-\e^{-2\alpha b}\right)\delta_\infty,
\]
see for instance \cite[Formula 2.0.2]{borodinsalminen2002}, and observe that the density part is decreasing in $\alpha$ as soon as $\alpha>\frac{1}{b}$. Thus, if we choose $A>\frac{1}{\beta+\eta}$, we have, for $t\in(0,\infty)$,
\[
\p\left(T_1^{u,[a]}=t\right)\leq\p\left(T_1^{u}=t\right)\footnote{Note that we also have $T_1^{u,[a]} \geq T_1^{u}$ almost surely!}.
\]
\noindent
Then, with $f^{u,[a]}(t)\de\p\left(T_1^{u,[a]}=t\right)$, one gets
\begin{align*}
	\p\left(Q_u^{[a]}>x\right)&=\int_{0}^{\infty} \p\left(Q_u^{[a]}>x\,\Big|\,T_1^{u,[a]}=t\right)\,\p\left(T_1^{u,[a]}=t\right)\dt\\
	&=\int_{0}^{\infty} \p\left(\sum_{k\geq 1} \ind_{T_k^{u,[a]}<+\infty}\Big(A+\eta T_k^{u,[a]}\Big)\e^{(\frac{\beta^2}{2}+\beta\eta)T_k^{u,[a]}+\beta A}>x\,\Bigg|\,T_1^{u,[a]}=t\right)\,f^{u,[a]}(t)\dt\\	
	&=\int_{0}^{\infty} \p\left(\sum_{k\geq 1} \ind_{T_k^{u}<+\infty}\big(A+\eta T_k^{u}\big)\e^{(\frac{\beta^2}{2}+\beta\eta)T_k^{u}+\beta A}>x\,\Bigg|\,T_1^{u}=t\right)\,f^{u,[a]}(t)\dt\\
	&\leq\int_{0}^{\infty} \p\left(\sum_{k\geq 1} \ind_{T_k^{u}<+\infty}\big(A+\eta T_k^{u}\big)\e^{(\frac{\beta^2}{2}+\beta\eta)T_k^{u}+\beta A}>x\,\Bigg|\,T_1^{u}=t\right)\,\p\left(T_1^{u}=t\right)\dt\\
	&=\p\left(Q_u>x\right).
\end{align*}
We thus have the following stochastic dominance
\begin{equation}\label{stocdom}
	Q_u^{[a]}\leq_s  Q_{u}\footnote{We didn't find a way to extend it to $Q^{[a]}\leq_s  Q$.}.
\end{equation}
And by linearity
	\[
	\es Q^{[a]}=\int_{\ft}  \es Q^{[a]}_u\,\mu(\mathrm du)\leq\int_{\ft}  \es Q_u \,\mu(\mathrm du)=\es Q.
	\]

	What remains to be proved is a uniform bound for $\mathrm{ess\,sup}\braket{Q^{[a]}}_\infty$ over $a$. The proof goes the same way as in \cite{lacoinrhodesvargas22}: with
	\[
	Q_u^{[a]}(t) \de \es\left[ Q_u^{[a]}\,\Big|\,\mathcal F_t\right],
	\]
	and
	\[
	\mathrm d Q_u^{[a]}(t) = A^{[a]}_u(t)\mathrm dX_t(u),
	\]
	the Markov property applied to $(a+X_t(u))_{t\geq0}$ yields to the same expression for the infinitesimal increment
	\[
	A^{[a]}_u(t) = \left\{
	\begin{array}{ll}
		f_1(t,a+X_t(u)), & \mbox{ if } t\in\left(T_k^{u,[a]},R_k^{u,[a]}\right), \\
		f_2(t,a+X_t(u)), & \mbox{ if } t\in\left(R_k^{u,[a]},T_{k+1}^{u,[a]}\right).
	\end{array}
	\right.
	\]

	And since in the first case $a+X_t(u)\geq \beta t$ and $a+X_t(u)\leq (\beta+\eta) t+A$ in the second case, one has by \cite[Lemma 5.3, Lemma 5.4]{lacoinrhodesvargas22}
	\[
	\left|A^{[a]}_u(t)\right|\leq C(t+1)\e^{(\frac{\beta^2}{2}+\beta\eta)t}.
	\]
	This yields to the desired uniform bound over $a$
	\[
	\braket{Q^{[a]}}_\infty \leq \int_{\ft^2\times [0,\infty)} C^2(t+1)^2\e^{(\beta^2+2\beta\eta)t} \,\ind_{t\leq u\wedge v}\, \,\mu(\mathrm du) \mu(\mathrm dv) \dt<\infty.
	\]
\end{proof}

\addcontentsline{toc}{section}{Appendix}
\appendix
\begin{appendix}


\section{Appendix}


The following lemma helps to control the terms in (\ref{decomp}).
\begin{lem}\label{MZ}
	If $p\in[1,\gamma\wedge2)$, then $Z_\be,W_\be\in L^p$ and there exists $B_{p}>0$ such that
	\[
	\es\left[\bigg|\sun \eb Z^u_\be\bigg|^p\,\Bigg|\,\fn \right]\leq B_{p}\,\es|Z_\be|^p \sun \e^{p(\be X_u-n\cb)},
	\]
	\[
	\es\left[\bigg|\sun \eb (X_u-\be n)(W^u_\be-1)\bigg|^p\,\Bigg|\,\fn \right]\leq B_{p}\,\es|\wb-1|^p \sun \e^{p(\be X_u-n\cb)}|X_u-\be n|^p.
	\]
\end{lem}
\begin{proof}
	If $Y_1,...,Y_N$ are centered and independent random variables in $L^p$, the Marcinkiewicz-Zygmund inequality provides a positive $B_p$ (which does not depend on $Y$) such that
	\[
	\es\left[\left|\sum_{k=1}^N Y_k\right|^p\right]\leq B_p\, \es\left[\Bigg(\sum_{k=1}^N Y_k^2\Bigg)^{p/2}\right].
	\]
	Now the function $x\mapsto x^{p/2}$ is subadditive since $p<2$, thus
	\[
	\es\left[\left|\sum_{k=1}^N Y_k\right|^p\right]\leq B_p\, \es\left[\sum_{k=1}^N |Y_k|^p\right],
	\]
	and the proof of the proposition is a consequence of this last inequality applied to the random variables $\eb Z_\beta^u$ and $\eb (X_u-\be n)(W^u_\be-1)$ conditionally on $\fn$.
\end{proof}

\begin{lem}\label{det}
There exists $\theta>0$ such that
\[
\log \det \Sigma_n = \theta 2^n - 2\log 2 + o(1).
\] 
\end{lem}

\begin{proof}
\begin{align*}
	\log\det\Sigma_n&=\log\left(2^n-1\right)+\sum_{k=0}^{n-1} 2^k\log\left(2^{n-k}-1\right)\\
	&=\log\left(2^n- 1\right) +2^n\sum_{k=1}^{n}\frac{1}{2^k}\left(k\log2+\log\left(1-\frac{1}{2^k}\right)\right)\\
	&=2^n\left(\log2\sum_{k=1}^{n}\frac{k}{2^k}+\sum_{k=1}^n\frac{1}{2^k}\log\left(1-\frac{1}{2^k}\right)+\frac{\log(2^n-1)}{2^n}\right)\\
	&=2^n\left(\theta\underbrace{-\log2\sum_{k=n+1}^{\infty}\frac{k}{2^k}-\sum_{k=n+1}^{\infty}\frac{1}{2^k}\log\left(1-\frac{1}{2^k}\right)+\frac{\log(2^n-1)}{2^n}}_{\frac{-2\log2}{2^n}+O(4^{-n})}\right),
\end{align*}
where we set $\theta = 2\log2+\sum_{k=1}^{\infty}\frac{1}{2^k}\log\left(1-\frac{1}{2^k}\right)$ and used $\sum_{k=n+1}^{\infty}\frac{k}{2^k} = \frac{n+2}{2^n}$ and \\$\sum_{k=n+1}^{\infty}\frac{1}{2^k}\log\left(1-\frac{1}{2^k}\right)=O(4^{-n})$.
\end{proof}

\end{appendix}

\section*{Acknowledgements}

The first author would like to thank Bastien Mallein for very stimulating discussions and for pointing out the scenario for the lower bound.

\addcontentsline{toc}{section}{References}

\bibliographystyle{abbrv}
\bibliography{biblio}

\end{document}